\documentclass{article}

\usepackage{amssymb,amsmath,color}
\usepackage{amsthm}
\usepackage{url}
\usepackage{tikz}
\usepackage{caption,subcaption}

\definecolor{red}{rgb}{1,0,0}
\definecolor{blue}{rgb}{.2,.2,.8}

\newtheorem{theorem}{Theorem}[section]
\newtheorem{corollary}[theorem]{Corollary}

\newtheorem{lemma}[theorem]{Lemma}

\theoremstyle{definition}
\newtheorem{example}{Example}

\begin{document}

\title{The partition function $p(n)$ in terms of the classical M\"{o}bius function}
\author{Mircea Merca\\
	\footnotesize Academy of Romanian Scientists\\
	\footnotesize Splaiul Independentei 54, Bucharest, 050094 Romania\\
	\footnotesize mircea.merca@profinfo.edu.ro
	\and Maxie D. Schmidt
	\\ 
	\footnotesize School of Mathematics, Georgia Institute of Technology\\
	\footnotesize Atlanta, GA 30332 USA\\
	\footnotesize maxieds@gmail.com, mschmidt34@gatech.edu
}
\date{}
\maketitle

\begin{abstract} 
In this paper, we investigate decompositions of the partition function $p(n)$ from the additive theory of partitions considering the famous M\"{o}bius function $\mu(n)$ from multiplicative number theory. Some combinatorial interpretations are given in this context. Our work extends several analogous identities proved recently relating $p(n)$ and Euler's totient function $\varphi(n)$. 
\end{abstract}

\section{Introduction} 

Very recently, the authors proved in \cite{MMMS} the following decomposition of the partition function $p(n)$:
$$p(n) =\frac{1}{2} \sum_{k=3}^{n+3} S^{(3)}_{n+3,k} \varphi(k),$$
where $S^{(r)}_{n,k}$ denotes the number of $k$'s in the partitions of $n$ 
with the smallest part at least $r$ and $\varphi(n)$ is Euler's totient function. 
This surprising result connects the famous classical totient function from multiplicative number theory with the function $p(n)$ 
from theory of partitions \cite{Andrews76}.

The aim of this paper is to prove similar expansions for the partition function $p(n)$ considering another famous object in multiplicative number theory: the M\"{o}bius function $\mu(n)$.
Recall that $\mu$ is defined for all positive integers $n$ and has its values in $\{-1, 0, 1\}$ depending on the factorization of $n$ into prime factors:
$$
\mu(n)=
\begin{cases}
0, & \text{if $n$ has a squared prime factor,}\\
(-1)^k, & \text{if $n$ is a product of $k\geqslant 0$ distinct primes.}
\end{cases}
$$
The classical M\"{o}bius function is central in multiplicative number theory and combinatorics. 
Here we prove two decomposition for the partition function $p(n)$ that combine $\mu(n)$ and our additive restricted partition function $S^{(r)}_{n,k}$ when $r\in\{1,2\}$.

\begin{theorem}
	\label{T1} 
	For $n\geqslant 0$, $r\in\{1,2\}$,
	\begin{equation*}\label{eqT1}
		p(n) = (-1)^{r-1} \sum_{k=r}^{n+r} S^{(r)}_{n+r,k} \mu(k).	
	\end{equation*}
\end{theorem}

\begin{example}
	We have $p(4)=5$ because the partitions in question are:
	$$ 4 = 3+1 = 2+2 = 2+1+1 = 1+1+1+1.$$
	For $r=1$, we consider the partitions of $5$:
	$$ 5 = 4+1 = 3+2 = 3+1+1 = 2+2+1 = 2+1+1+1 = 1+1+1+1+1.$$
	We have
	$$\mu(1)\cdot 12 +\mu(2)\cdot 4+ \mu(3) \cdot 2 + \mu(4) \cdot 1 + \mu(5) \cdot 1 = 12 - 4 - 2 + 0 - 1 = 5.$$
	For $r=2$, the partitions of $6$ that do not contain $1$ as a part are:
	$$ 6 = 4+2 = 3+3 = 2+2+2 .$$
	We also have
	$$-\mu(2)\cdot 4 - \mu(3) \cdot 2 - \mu(4) \cdot 1 - \mu(6) \cdot 1 = 4 + 2 + 0 - 1 = 5.$$
\end{example}

The set of partitions of $n$ containing $r$ as a part can be obtained from the set of
unrestricted partitions of $n-r$ by adding to each partition a single $r$. This is an  example of a
bijection between two sets of partitions. For this reason, we can say that our theorem
establishes a connection between the set of partitions of $n$ that contain $r$ as a part, and
the set of partitions of $n$ with the smallest part at least $r$ when $r\in\{1,2\}$.

\begin{corollary}
	For $n\geqslant 0$,  $r\in\{1,2\}$,
	$$p(n-r)=(-1)^{r-1}\sum_{rt_r+(r+1)t_{r+1}+\cdots+nt_n=n} \mu(r) t_r+\mu(r+1)t_{r+1} +\cdots +\mu(n) t_n.$$
\end{corollary}

Upon reflection, one expects that there might be an infinite family of such identities where Theorem \ref{T1} represents the first and the second entries.

\begin{theorem}\label{Th2}
	For $n\geqslant 0$, $r\geqslant 1$, we have the identity
	\begin{equation*}\label{eqTh2}
		\sum_{k=r}^{n+r} S^{(r)}_{n+r,k} \mu(k) = \sum_{j\geqslant 0} a_{r,j} p(n-j),
	\end{equation*}
	where the coefficients $a_{r,j}$ are given by the recurrence relations
	\begin{align*}
		a_{r+1,j} = a_{r,j+1} - \mu(r) b_{r,j+1}
		-\begin{cases}
			0,& \text{for $j+1<r$,}\\
			a_{r,j+1-r},& \text{for $j+1\geqslant r$,}
		\end{cases}	
	\end{align*}
	and
	$$
	b_{r+1,j} = b_{r,j} 
	-\begin{cases}
	0,& \text{for $j<r$,}\\
	b_{r,j-r},& \text{for $j\geqslant r$,}
	\end{cases}
	$$
	with the initial conditions:
	$$a_{1,j}=b_{1,j}=\begin{cases}
	1,& \text{for $j=0$,}\\
	0,& \text{for $j>0$.}
	\end{cases}$$
\end{theorem}
By Theorem \ref{Th2}, we see that the case $r=2$ of Theorem \ref{T1} can be derived as a corollary of the case $r=1$.
We remark that for $r>1$ the coefficients $b_{r,j}$ satisfy the relation
$$b_{r,j} = b_{r-1,j} - (-1)^r \cdot b_{r-1,r(r-1)/2-j}.$$

\begin{figure}[h!]
	\begin{minipage}{\linewidth} 
		\begin{center} 
			\footnotesize
			\begin{equation*} 
				\boxed{ 
					\begin{array}{r|rrrrrrrrrrrrrrrrrrrrr} 
						& 0 & 1 & 2 & 3 & 4 & 5 & 6 & 7 & 8 & 9 & 10 & 11 & 12 & 13 & 14 & 15\\
						\hline 
						1 & 1 & 0 & 0 & 0 & 0 & 0 & 0 & 0 & 0 & 0 & 0 & 0 & 0 & 0 & 0 & 0 &\\
						2 & 1 & -1 & 0 & 0 & 0 & 0 & 0 & 0 & 0 & 0 & 0 & 0 & 0 & 0 & 0 & 0 & \\ 
						3 & 1 & -1 & -1 & 1 & 0 & 0 & 0 & 0 & 0 & 0 & 0 & 0 & 0 & 0 & 0 & 0 & \\
						4 & 1 & -1 & -1 & 0 & 1 & 1 & -1 & 0 & 0 & 0 & 0 & 0 & 0 & 0 & 0 & 0 &  \\
						5 & 1 & -1 & -1 & 0 & 0 & 2 & 0 & 0 & -1 & -1 & 1 & 0 & 0 & 0 & 0 & 0 &  \\
						6 & 1 & -1 & -1 & 0 & 0 & 1 & 1 & 1 & -1 & -1 & -1 & 0 & 0 & 1 & 1 & -1 &\\
						\hline
					\end{array}
				}
			\end{equation*}
		\end{center} 
		\subcaption*{(i) $b_{r,j}$} 
	\end{minipage} 
	\begin{minipage}{\linewidth} 
		\begin{center} 
			\footnotesize 
			\begin{equation*} 
				\boxed{ 
					\begin{array}{r|rrrrrrrrrrrrrrrrrrrrr} 
						& 0 & 1 & 2 & 3 & 4 & 5 & 6 & 7 & 8 & 9 & 10 & 11 & 12 & 13 & 14 & 15\\
						\hline 
						1 & 1 & 0 & 0 & 0 & 0 & 0 & 0 & 0 & 0 & 0 & 0 & 0 & 0 & 0 & 0 & 0 \\
						2 &-1 & 0 & 0 & 0 & 0 & 0 & 0 & 0 & 0 & 0 & 0 & 0 & 0 & 0 & 0 & 0 \\
						3 & -1 & 1 & 0 & 0 & 0 & 0 & 0 & 0 & 0 & 0 & 0 & 0 & 0 & 0 & 0 & 0 \\
						4 & 0 & -1 & 2 & -1 & 0 & 0 & 0 & 0 & 0 & 0 & 0 & 0 & 0 & 0 & 0 & 0 \\
						5 & -1 & 2 & -1 & 0 & 1 & -2 & 1 & 0 & 0 & 0 & 0 & 0 & 0 & 0 & 0 & 0  \\
						6 & 1 & -2 & 0 & 1 & 1 & -1 & 1 & -1 & -2 & 3 & -1 & 0 & 0 & 0 & 0 & 0  \\
						7 & -1 & 1 & 1 & 1 & -2 & -1 & 0 & -1 & 3 & -1 & 1 & -1 & 0 & 1 & -2 & 1  \\
						\hline 
					\end{array}
				} 
			\end{equation*} 
		\end{center} 
		\subcaption*{(ii) $a_{r,j}$} 
	\end{minipage}
	\caption{The coefficients $a_{r,j}$ and $b_{r,j}$}
	\label{F1} 
\end{figure} 

Theorem \ref{Th2} results directly from the following recurrence relation for $S^{(r)}_{n,k}$. 

\begin{lemma}\label{L1}
	For $n\geqslant k> r \geqslant 1$,
	$$  S^{(r+1)}_{n+r,k} = S^{(r)}_{n+r,k} - S^{(r)}_{n,k}.$$
\end{lemma}

\begin{proof}
	The number of $k$'s in the partitions of $n+r$ with the smallest part exactly $r$ is given by
	$$S^{(r)}_{n+r,k}-S^{(r+1)}_{n+r,k}.$$
	On the other hand, the set of partitions of $n+r$ with the smallest part exactly $r$  can be obtained from the set of unrestricted partitions of $n$ by adding to each partition a single $r$. For $k>r$, it is clear that the number of $k$'s in the partitions of $n+r$ with the smallest part exactly $r$ is
	$ S^{(r)}_{n,k}$	
	and the lemma is proved.
\end{proof}

There is a natural question related to Theorem \ref{Th2}: is it possible to have combinatorial interpretations for sums such as those in Theorem \ref{Th2} when $r>2$?
The cases $r\in\{3,4,5\}$ are investigated in this paper. As usual, the $n$th order backward difference is denoted by
$$\nabla^n [f](x) = \sum_{k=0}^n (-1)^k \binom{n}{k} f(x-k).$$

\begin{corollary}
	\label{T2} 
	For $n\geqslant 0$, the partitions of $n$ with no parts equal to $1$ are counted by  
	$$ 
	\nabla [p](n) = -\sum_{k=3}^{n+3} S^{(3)}_{n+3,k} \mu(k).$$
\end{corollary}

\begin{example}
	We have already seen above that the integer $6$ has $4$ partitions with no parts equal to $1$.
	The partitions of $9$ with the smallest part at least $3$ are:
	$$ 9 = 6+3 = 5+4 = 3+3+3.$$
	Considering that $\mu(4)=\mu(9)=0$, we have
	$$-\mu(3)\cdot 4- \mu(5) \cdot 1 - \mu(6) \cdot 1  = 4  + 1 - 1   = 4.$$
\end{example}

\begin{corollary}
	\label{T3} 
	For $n\geqslant 3$, the partitions of $n$  with no parts equal to $1$ and the largest part occurring more than once are counted by
	$$ 
	\nabla^2 [p](n) = -\sum_{k=4}^{n+5} S^{(4)}_{n+5,k} \mu(k).$$
\end{corollary}

\begin{example}
	The integer $8$ has $3$ partitions with no parts equal to $1$ and the largest part occurring more than once:
	$$4+4=3+3+2=2+2+2+2.$$
	The partitions of $13$ with the smallest part at least $4$ are:
	$$ 13 = 9+4 = 8+5 = 7+6 = 5+4+4.$$
	We consider only the parts that are products of distinct primes and obtain:
	$$ - \mu(5) \cdot 2 - \mu(6) \cdot 1 - \mu(7) \cdot 1 - \mu(13) \cdot 1
	= 2-1+1+1  = 3.$$
\end{example}

\begin{corollary}
	\label{T4}
	For $n\geqslant 6$, the partitions of $n$ with no parts equal to $1$, at most one part equal to $2$, and the largest part occurring more than once are counted by
	$$\nabla^2 [p](n)-\nabla^2 [p](n-4)=-\sum_{k=5}^{n+5} S^{(5)}_{n+5,k} \mu(k).$$
\end{corollary}

\begin{example}
	The integer $12$ has $4$ partitions with no parts equal to $1$, at most one part equal to $2$, and the largest part occurring more than once:
	$$6+6=5+5+2=4+4+4=3+3+3+3.$$
	The partitions of $17$ with the smallest part at least $5$ are:
	$$ 17 = 12+5 = 11+6 = 10+7 = 9+8 = 7+5+5 = 6+6+5.$$
	Considering only the parts that are products of distinct primes, we have:
	$$-\mu(5)\cdot 4 - \mu(6) \cdot 3 - \mu(7) \cdot 2 - \mu(10) \cdot 1 - \mu(11) \cdot 1- \mu(17) \cdot 1\\
	= 5 -3 +2 -1+1 = 4.$$	
\end{example}

We will prove these corollaries in Section \ref{S3}.

\section{Proof of Theorem \ref{Th2}}

First, we prove the case $r=1$. As we found in \cite{MMMS}, the generating function of $S^{(r)}_{n,k}$ is given by
\begin{align}
	\sum_{n=k}^{\infty} S^{(r)}_{n,k} q^n 
	= \frac{q^k}{1-q^k} \cdot \frac{1}{(q^r; q)_{\infty}},\label{eqS}
\end{align}
where $q \in \mathbb C$, $|q|<1$, and the usual $q$-Pochhammer symbol is defined by
$$(a;q)_\infty = \prod_{k=0}^{\infty} (1-aq^k).$$
Then we can write
\begin{equation} \label{eq1a}
	\frac{q}{(q;q)_\infty}  = \frac{1}{(q;q)_\infty} \sum_{k=1}^{\infty} \frac{\mu(k) q^k}{1-q^k} = \sum_{n=1}^{\infty} \sum_{k=1}^{\infty}  \mu(k) S^{(1)}_{n,k} q^n, 
\end{equation}
where we have invoked the well-known Lambert series generating function \cite[Theorem 263]{Hardy}
$$\sum_{k=1}^{\infty} \frac{\mu(k) q^k}{1-q^k} = q.$$
On the other hand, we have
\begin{align}
	\frac{q}{(q;q)_\infty} = \sum_{n=1}^{\infty} p(n-1) q^{n}, \label{eq2a}
\end{align}
where we have used Euler's well-known partition generating function formula:
$$\sum_{n=0}^\infty p(n) q^n = \frac{1}{(q;q)_\infty}.$$
By \eqref{eq1a} and \eqref{eq2a}, we deduce that
\begin{equation}
	\label{eqn_pnm1_Snk1_rel}
	p(n-1) = \sum_{k=1}^{n}  \mu(k) S^{(1)}_{n,k}, 
\end{equation} 
and the case $r=1$ is proved.

According to Lemma \ref{L1}, we have
\begin{align*}
	&\sum_{j\geqslant 0} a_{r+1,j} p(n-j) \\
	& \quad=\sum_{k=r+1}^{n+r+1} S^{(r+1)}_{n+r+1,k} \mu(k) \\
	& \quad= \sum_{k=r+1}^{n+1+r} S^{(r)}_{n+1+r,k} \mu(k) - \sum_{k=r+1}^{n+1+r} S^{(r)}_{n+1,k} \mu(k)\\
	& \quad= \sum_{k=r}^{n+1+r} S^{(r)}_{n+1+r,k} \mu(k) - \sum_{k=r}^{n+1} S^{(r)}_{n+1,k} \mu(k) - \left( S^{(r)}_{n+1+r,r}-S^{(r)}_{n+1,r} \right) \mu(r)\\ 
	& \quad= \sum_{j\geqslant 0} a_{r,j} p(n+1-j) - \sum_{j\geqslant 0} a_{r,j} p(n+1-r-j)-\mu(r)\sum_{j\geqslant 0} b_{r,j} p(n+1-j),
\end{align*}	
where we have used that
$$S^{(r)}_{n+r,r}-S^{(r)}_{n,r}=\sum_{j\geqslant 0} b_{r,j} p(n-j)$$
counts the partitions of $n$ into parts greater than or equal to $r$. 
In other words, $b_{r,j}$ is the coefficient of $q^j$ in $(q;q)_{r-1}$. 
So we have
\begin{align*}
	\sum_{j\geqslant 0} b_{r+1,j} q^j  
	= (1-q^r) (q;q)_{r-1} 
	= (1-q^r) \sum_{j\geqslant 0} b_{r,j} q^j
	= \sum_{j\geqslant 0} b_{r,j} q^j	- \sum_{j\geqslant 0} b_{q,j} q^{r+j}.
\end{align*}
Equating the coefficients of $q^j$ in this relation gives the recurrence relation for $b_{r,j}$.
This completes the proof.

\section{Proofs of Corollaries}
\label{S3}

The coefficients $a_{r,j}$ used in these corollaries are listed in Figure \ref{F1}. We remark that these corollaries can be proved independently of Theorem \ref{Th2} using the method of generating functions. 
We will illustrate this fact for Corollary \ref{T2}.

\subsection{Proof of Corollary \ref{T2}.}
It is well known that the first difference of the partition function $$\nabla [p](n)=p(n)-p(n-1)$$
counts the partitions of $n$ with no parts equal to $1$. 
The proof follows easily considering the case $r=3$ of Theorem \ref{Th2}.

In addition, we provide a proof of Corollary \ref{T2} involving the method of generating functions. By \eqref{eqS}, with $r$ replaced by $3$, we can write
\begin{align*}
	& \frac{(1-q)(1-q^2)}{(q;q)_\infty}\sum_{k=1}^{\infty} \frac{\mu(k) q^k}{1-q^k} \\
	& \qquad= \frac{(1-q)(1-q^2)}{(q;q)_\infty} \left(\frac{q}{1-q}-\frac{q^2}{1-q^2} \right) + \sum_{k=3}^{\infty} \sum_{n=3}^{\infty} \mu(k) S^{(3)}_{n,k} q^n \\
	& \qquad= \frac{q-q^2}{(q;q)_\infty} + \sum_{k=3}^{\infty} \sum_{n=3}^{\infty} \mu(k) S^{(3)}_{n,k} q^n 
\end{align*}
and 
\begin{align*}
	\frac{(1-q)(1-q^2)}{(q;q)_\infty}\sum_{k=1}^{\infty} \frac{\mu(k) q^k}{1-q^k} 
	= \frac{q-q^2}{(q;q)_\infty}-\frac{q^3-q^4}{(q;q)_\infty}.
\end{align*}
It follows that 
$$ \sum_{n=3}^\infty \left(\sum_{k=3}^{n} \mu(k) S^{(3)}_{n,k} \right) q^n 
= -\frac{q^3-q^4}{(q;q)_\infty}=-\sum_{n=3}^\infty \nabla [p](n-3) q^{n},$$
where we have invoked the generating function for $p(n)$, completing the proof.

\subsection{Proof of Corollary \ref{T3}}

The conjugate of a partition with no parts equal to $1$ is a partition with the largest part occurring more than once.
We denote by $\mathcal{A}_1(n)$ the set of the partitions of $n$ with the largest part occurring more than once. This set is the union of the following disjoint sets:
\begin{enumerate}
	\item[] $\mathcal{B}_1(n)$: the set of partitions of $n$ with the smallest part exactly $1$ and the largest part occurring more than once.
	\item[] $\mathcal{C}_1(n)$: the set of the partitions of $n$ with no parts equal to $1$ and the largest part occurring more than once.
\end{enumerate}
The set  $\mathcal{B}_1(n)$ can be obtained from the set  $\mathcal{A}_1(n-1)$ by adding to each partition a single $1$. It is clear that
\begin{align*}
	|\mathcal{C}_1(n)| &= |\mathcal{A}_1(n)| - |\mathcal{A}_1(n-1)|
	= \nabla[p](n) - \nabla[p](n-1) = \nabla^2 [p](n).
\end{align*}
The proof follows easily considering the case $r=4$ of Theorem \ref{Th2}.

\subsection{Proof of Corrolary \ref{T4}}

We denote by $\mathcal{A}_2(n)$ the set of the partitions of $n$ with no parts equal to $1$ and the largest part occurring more than once. This set is the union of the following disjoint sets:
\begin{enumerate}
	\item[] $\mathcal{B}_2(n)$: the set of partitions of $n$ with the smallest part exactly $2$ in which the smallest part and the largest part occurring more than once.
	\item[] $\mathcal{C}_2(n)$: the set of the partitions of $n$ with no parts equal to $1$, at most one part equal to $2$, and the largest part occurring more than once.
\end{enumerate}

The set  $\mathcal{B}_2(n)$ can be obtained from the set  $\mathcal{A}_2(n-4)$ by adding to each partition two parts equal to $2$. It is clear that
\begin{align*}
	|\mathcal{C}_2(n)| &= |\mathcal{A}_2(n)| - |\mathcal{A}_2(n-4)|
	= \nabla^2[p](n) - \nabla^2[p](n-4).
\end{align*}
The proof follows easily considering the case $r=5$ of Theorem \ref{Th2}.

\section{Conclusions} 

A new method for generating combinatorial identities has been introduced in the paper by Theorem \ref{Th2}. 
This theorem can be generalized in many ways. The M\"{o}bius function $\mu(n)$ can be replaced by any arithmetic function that allows us to prove the case $r=1$. For example, replacing the M\"{o}bius function $\mu(n)$ by the Euler totient function $\varphi(n)$, we can obtain 
identities involving the partition function $p(n)$:
\begin{align}
	& \sum_{k=1}^{n+1} S^{(1)}_{n+1,k} \varphi(k) = \sum_{k=0}^n (k+1) p(n-k)\label{eqC1}\\
	& \sum_{k=2}^{n+2} S^{(2)}_{n+2,k} \varphi(k) = \sum_{k=0}^n p(n-k)\label{eqC2}.
\end{align}
These sums have some combinatorial interpretations. The sums in \eqref{eqC1} give the sum of parts, counted without multiplicity, in all partitions of $n+1$ \cite[A014153]{OEIS}. The sums in \eqref{eqC2} count the number of $1$'s in all partitions of $n+1$, that is, $S^{(1)}_{n+1,1}$. Note that the identity \eqref{eqC2} has been recently proved in \cite[Theorem 1]{MMMS}.


Finally, we give two new identities involving the  M\"{o}bius function $\mu(n)$, the partition function $p(n)$ and prime numbers.

\begin{corollary}
	For $n \geqslant 0$ and $\alpha$ a prime number, we have the following identities:
	\begin{enumerate}
		\item[(i)] $\displaystyle{\sum_{k=1}^{n+1} S^{(1)}_{n+1,k} \mu(\alpha k) = - \sum_{k\geqslant 0} p(n+1-\alpha^k)}$
		\item[(ii)] $\displaystyle{ \sum_{k=2}^{n+2} S^{(2)}_{n+2,k} \mu(\alpha k) =  \sum_{k\geqslant 0} \big( p(n+1-\alpha^k) - p(n+2-\alpha^{k+1}) \big) }$.
	\end{enumerate}
\end{corollary}

\begin{proof}
	Considering that
	$$\sum_{n=0}^\infty \frac{\mu(\alpha n) q^n}{1-q^n} = -\sum_{n=0}^{\infty} q^{\alpha^n},$$
	the proof of the first identity is similar to the proof of the case $r=1$ of Theorem \ref{Th2}.
	Invoking Lemma \ref{L1}, the second identity follows as a consequence of the first identity. 	
\end{proof} 

The formula in (i) of the corollary can be expressed in the alternate form  
\[
\sum_{k=1}^{\left\lceil \frac{n+1}{\alpha} \right\rceil} \sum_{d=1}^{\alpha-1} 
S_{n+1,\alpha (k-1)+d}^{(1)} \mu(\alpha (k-1)+d) = 
\sum_{j=0}^{\lfloor \log_{\alpha}(n+1) \rfloor} p(n+1-\alpha^j), 
\] 
for any prime $\alpha$. 
Similarly, the expansion on the right-hand side of the identity in (ii) above may be written as 
a finite sum with respect to $n$. 

For further reading at the intersection of the additive and multiplicative branches of number theory, we recommend \cite{Alladi,Jameson,Merca,Schneider,Wakhare}.

\paragraph{Acknowledgments.}  
The authors thank the referees for their helpful comments.

\bigskip


\end{document}